\documentclass[11pt]{amsart}

\usepackage{amsmath,amssymb,amsthm}
\usepackage{xcolor}
\usepackage[colorlinks=true,linkcolor=blue!60!black,citecolor=blue!60!black,urlcolor=blue!60!black]{hyperref}
\usepackage{microtype}

\numberwithin{equation}{section}

\theoremstyle{plain}
\newtheorem{theorem}{Theorem}[section]
\newtheorem{proposition}[theorem]{Proposition}
\newtheorem{lemma}[theorem]{Lemma}
\newtheorem{corollary}[theorem]{Corollary}
\newtheorem*{thmA}{Theorem A}
\newtheorem*{thmB}{Theorem B}
\newtheorem*{thmC}{Theorem C}

\theoremstyle{definition}
\newtheorem{definition}[theorem]{Definition}
\newtheorem{example}[theorem]{Example}
\newtheorem{question}[theorem]{Question}

\theoremstyle{remark}
\newtheorem{remark}[theorem]{Remark}

\DeclareMathOperator{\Hom}{Hom}

\DeclareMathOperator{\Ext}{Ext}
\DeclareMathOperator{\Tor}{Tor}
\DeclareMathOperator{\HH}{HH}
\DeclareMathOperator{\HHl}{H}
\DeclareMathOperator{\Tr}{Tr}
\DeclareMathOperator{\tr}{tr}
\DeclareMathOperator{\gldim}{gl.dim}
\DeclareMathOperator{\pdim}{pd}
\DeclareMathOperator{\rad}{rad}

\DeclareMathOperator{\im}{im}
\DeclareMathOperator{\coker}{coker}
\DeclareMathOperator{\modcat}{mod}

\DeclareMathOperator{\Dcat}{\mathcal{D}}
\DeclareMathOperator{\Z}{Z}
\DeclareMathOperator{\B}{B}

\newcommand{\D}{\mathsf{D}}
\newcommand{\E}{E}
\newcommand{\kk}{k}
\newcommand{\Lm}{\Lambda}
\newcommand{\Lme}{\Lambda^{\mathrm{e}}}
\newcommand{\om}{\omega}
\newcommand{\Pii}{\Pi}
\newcommand{\Ltimes}{\otimes^{\mathbf{L}}}

\begin{document}

\title[$\tau$-Hochschild theory and the Coxeter automorphism]{$\tau$-Hochschild (co)homology, the square of the Serre bimodule, and the Coxeter automorphism of the Tamarkin--Tsygan calculus}

\author{Marco Armenta}

\begin{abstract}
We relate two recent enrichments of the Hochschild theory of a finite-dimensional algebra $\Lm$: the $\tau$-Hochschild (co)homology of Cibils, Lanzilotta, Marcos and Solotar, built from Iyama's higher Auslander--Reiten translates of the regular bimodule, and the Coxeter automorphism $\sigma_\Lm$ of the Tamarkin--Tsygan calculus. We show that the Nakayama functor of the enveloping algebra transforms Happel's minimal resolution into a complex representing $\D\Lm\Ltimes_\Lm \D\Lm$, the square of the Serre bimodule whose shift generates $\sigma_\Lm$, and that the $\tau$-translates $\tau_n\Lm$ are precisely the cycle bimodules of this complex. This produces extensions $0\to \B_n\to \tau_n\Lm\to \Tor_n^\Lm(\D\Lm,\D\Lm)\to 0$ whose outer term is dual to $\Ext^n_{\Lme}(\Lm,\Lme)$ and whose inner term is a strictly Morita-theoretic residue of the minimal model. In top degree $d=\gldim\Lm$ the residue vanishes and $\tau_d\Lm$ is the dual of the degree-one component of the $(d+1)$-preprojective algebra of Iyama--Oppermann; for $\Lm=\kk Q$ hereditary, $\tau_{\Lme}\Lm\cong \D\Pi(Q)_1$ and $\HH^1_\tau(\kk Q)$ is the degree-one part of the zeroth Hochschild homology of the preprojective algebra. For selfinjective algebras the derived part vanishes identically, which explains structurally the growth of $\tau$-cohomology for the Buchweitz--Green--Madsen--Solberg algebras. Taking Euler characteristics in the Cibils--Lanzilotta--Marcos--Solotar dimension formulas recovers Happel's trace formula $\sum_i(-1)^i\dim\HH^i(\Lm)=-\tr\sigma_\Lm$. We prove that the two refinements are transversal, propose the combined Morita invariant, exhibit derived-equivalent algebras of finite global dimension whose $\tau$-translates have identical dimension but opposite composition, and pose the problem of derived invariance of $\tau$-Hochschild theory over the smooth locus.
\end{abstract}

\maketitle

\section{Introduction}\label{sec:intro}

Two recent and independent lines of work enrich the Hochschild theory of a finite-dimensional algebra $\Lm$ over a field $\kk$ in ways that, at first sight, point in opposite directions.

On one side, Cibils, Lanzilotta, Marcos and Solotar \cite{CLMS1,CLMS2} introduce the \emph{$\tau$-Hochschild cohomology and homology} of $\Lm$. Following one of the guiding ideas of $\tau$-tilting theory \cite{AIR} replace $\Ext^1_A(M,N)$ by $\D\Hom_A(N,\tau M)$ they apply Iyama's higher Auslander--Reiten translates $\tau_n=\tau\Omega^{n-1}$ \cite{Iya1,Iya2} to the regular bimodule $\Lm$ over the enveloping algebra $\Lme$, computed on Happel's minimal projective resolution \cite{Hap1}, and set
\[
\HHl^n_\tau(\Lm,X)=\D\Hom_{\Lm\text{-}\Lm}(X,\tau_n\Lm),\qquad
\HHl_n^\tau(\Lm,X)=\Hom_{\Lm\text{-}\Lm}(\D X,\tau_n\Lm).
\]
These theories interpolate between Hochschild theory and the minimal model of $\Lm$: cohomologically one keeps cochains modulo coboundaries (forgetting the cocycle condition), homologically one keeps cycles (forgetting boundaries) \cite[Theorem~2.12]{CLMS2}. They are Morita invariant but, remarkably, \emph{not} derived invariant \cite[Example~4.12]{CLMS2}, and they detect refined versions of the finiteness questions of Happel and Han \cite[Theorem~5.5]{CLMS2}.

On the other side, it has been shown \cite{A1,A2,A3,A4,A5} that the full Tamarkin--Tsygan calculus $\HHl(\Lm)$, that is, Hochschild cohomology and homology together with cup, cap, Gerstenhaber bracket, Lie actions and Connes' operator, is a derived invariant, and that for $\gldim\Lm<\infty$ the two-sided tilting complex $\D\Lm[-1]$ induces a distinguished automorphism $\sigma_\Lm$ of this calculus, the \emph{Coxeter automorphism}. The pair $(\HHl(\Lm),\sigma_\Lm)$ is a derived invariant which categorifies the Coxeter polynomial and strictly refines the calculus alone \cite[Theorem~4.2, Theorem~5.5]{A5}.

So one theory refines Hochschild theory \emph{transversally to} derived equivalence, the other refines it \emph{along} derived equivalence. The purpose of this manuscript is to show that the two constructions are nevertheless windows onto a \emph{single object}: the image of Happel's minimal resolution under the Nakayama functor of the enveloping algebra, which is a complex representing the square $\om_\Lm\Ltimes_\Lm\om_\Lm$ of the Serre bimodule $\om_\Lm=\D\Lm$, whose shift generates $\sigma_\Lm$. Once this is in place, each theory illuminates the other: the $\tau$-translates acquire a derived-invariant ``shadow'' and a strictly Morita-theoretic ``residue''; the Coxeter trace formula of Happel becomes the Euler characteristic of the $\tau$-correction terms; and the top $\tau$-translate is identified with higher preprojective data in the sense of Iyama--Oppermann \cite{IO}. 

We work throughout over a field $\kk$, with $\Lm$ a finite-dimensional $\kk$-algebra whose semisimple quotient $\E=\Lm/\rad\Lm$ is separable, as in \cite{CLMS2}. Write $\D=\Hom_\kk(-,\kk)$, $\Lme=\Lm\otimes_\kk\Lm^{\mathrm{op}}$, identify $\Lm$-bimodules with left $\Lme$-modules, and let $\nu=\D\Hom_{\Lme}(-,\Lme)$ be the Nakayama functor of $\modcat\Lme$. Let $P_\bullet\to\Lm$ be Happel's minimal projective resolution of the regular bimodule, with $P_n=\Lm\otimes_\E T_n\otimes_\E\Lm$ and $T_n=\Tor_n^\Lm(\E,\E)$ \cite{Hap1,BK,CLMS2}.

\begin{thmA}[= Theorem~\ref{thm:bridge}]
Let $\Lm$ be a finite-dimensional algebra with $\E$ separable.
\begin{enumerate}
\item There are natural isomorphisms of complexes of bimodules
\[
\nu(P_\bullet)\;\cong\;\D\Lm\otimes_\Lm P_\bullet\otimes_\Lm\D\Lm,\qquad \nu(P_n)\cong \D\Lm\otimes_\E T_n\otimes_\E\D\Lm,
\]
and $\nu(P_\bullet)$ represents $\om_\Lm\Ltimes_\Lm\om_\Lm=\D\Lm\Ltimes_\Lm\D\Lm$ in $\Dcat(\Lme)$.
\item For every $n\geq 0$ there are natural isomorphisms of bimodules
\[
\HHl_n\bigl(\nu P_\bullet\bigr)\;\cong\;\Tor_n^\Lm(\D\Lm,\D\Lm)\;\cong\;\D\Ext^n_{\Lme}(\Lm,\Lme)\;\cong\;\D\Ext^n_{\Lm}(\D\Lm,\Lm).
\]
\item For every $n\geq 1$, the $n$-th higher Auslander--Reiten translate of the regular bimodule is the bimodule of $n$-cycles: $\tau_n\Lm=\Z_n(\nu P_\bullet)=\ker\nu(d_n)$. In particular there are short exact sequences of bimodules
\[
0\longrightarrow \B_n(\nu P_\bullet)\longrightarrow \tau_n\Lm \longrightarrow \Tor_n^\Lm(\D\Lm,\D\Lm)\longrightarrow 0,
\]
where $\B_n(\nu P_\bullet)=\im\nu(d_{n+1})$.
\end{enumerate}
\end{thmA}

Thus the $\tau$-Hochschild theory of \cite{CLMS2} is computed by the \emph{cycles} of the canonical minimal-model representative of the square of the Serre bimodule. The homology of this representative, which is the Serre-twisted Hochschild homology $\Tor_\bullet^\Lm(\D\Lm,\D\Lm)$, dual to the cohomology $\Ext^\bullet_{\Lme}(\Lm,\Lme)$ of the inverse dualizing complex, depends only on the quasi-isomorphism class of $\om_\Lm^{\Ltimes 2}$, an object whose conjugation class is a derived invariant by \cite[Lemma~3.4]{A5}; the boundary subobject $\B_n$ is a strictly Morita-theoretic residue of the minimal model. Two structural consequences are immediate and, we believe, clarifying. First, if $\Lm$ is selfinjective, then $\Lme$ is selfinjective, so the derived part vanishes identically and $\tau_n\Lm=\B_n$ is pure residue for all $n\geq 1$ (Corollary~\ref{cor:selfinj}); the unbounded growth of the $\tau$-cohomology of the Buchweitz--Green--Madsen--Solberg algebras \cite{BGMS} computed in \cite[Section~6]{CLMS2} is therefore carried entirely by the minimal model, with zero derived shadow. Second, in top degree the residue vanishes:

\begin{thmB}[= Theorem~\ref{thm:preproj}]
Let $\gldim\Lm=d\geq 1$ and let $\Pii_{d+1}(\Lm)=T_\Lm\bigl(\Ext^d_{\Lme}(\Lm,\Lme)\bigr)$ be the $(d+1)$-preprojective algebra of Iyama--Oppermann \cite{IO}, with its tensor grading. Then:
\begin{enumerate}
\item $\tau_d\Lm\;\cong\;\D\bigl(\Pii_{d+1}(\Lm)_1\bigr)$ as bimodules;
\item $\HH^d_\tau(\Lm)\;\cong\;\HH_0\bigl(\Lm,\Pii_{d+1}(\Lm)_1\bigr)=\bigl(\Pii_{d+1}(\Lm)/[\Pii_{d+1}(\Lm),\Pii_{d+1}(\Lm)]\bigr)_1$;
\item $\HH_d^\tau(\Lm)\;\cong\;\Hom_{\Lme}\bigl(\Pii_{d+1}(\Lm)_1,\Lm\bigr)=0$, the vanishing being equivalent to the top-degree case $\HH_d(\Lm)=0$ of the Han--Keller vanishing theorem.
\end{enumerate}
In particular, for $\Lm=\kk Q$ hereditary, $\tau_{\Lme}\Lm\cong\D\Pi(Q)_1$ for the classical preprojective algebra $\Pi(Q)$, and $\HH^1_\tau(\kk Q)\cong\HH_0(\Pi(Q))_1$.
\end{thmB}

In the same degree both $\tau$-theories collapse onto the ordinary ones for arbitrary coefficients (Proposition~\ref{prop:topcollapse}); this explains a priori the vanishing of the degree-one \emph{excess} of hereditary algebras computed in \cite{CLMS1}, and exhibits the top Hochschild cohomology functor as $\Ext^d_{\Lme}(\Lm,\Lme)\otimes_{\Lme}(-)$, a degenerate, hypothesis-free top edge of Van den Bergh duality \cite{VdB}.

The third point of contact is numerical. Specializing the dimension formulas of \cite[Theorem~4.5]{CLMS2} beyond the top degree and taking Euler characteristics, we obtain:

\begin{thmC}[= Theorem~\ref{thm:trace}]
Let $\Lm$ be elementary with $\gldim\Lm<\infty$ and Cartan matrix $C=(c_{xy})$, $c_{xy}=\dim_\kk e_y\Lm e_x$. Then
\begin{gather*}
\sum_{i\geq 0}(-1)^i\dim_\kk\HH^i(\Lm)\;=\;\tr\bigl(C^{-1}C^{T}\bigr)\;=\;-\tr\bigl(\sigma_{\Lm}|_{\HH_\bullet(\Lm)}\bigr),\\
\sum_{i\geq 0}(-1)^i\dim_\kk\HH_i(\Lm)\;=\;\tr\bigl(C^{-1}C\bigr)\;=\;n,
\end{gather*}
where $n$ is the number of simple modules.
\end{thmC}

The first identity is Happel's trace formula \cite{Hap2}, lifted to the operator level in \cite[Corollary~5.9]{A5} and categorified in degree zero by Han's Lefschetz formula \cite{Han2}; the point here is that it is precisely the Euler-characteristic shadow of the $\tau$-correction terms of \cite{CLMS2}. The three results triangulate: the individual $\tau$-groups refine Happel's identity degreewise, the Coxeter automorphism refines it at the operator level, and the identity itself is their common numerical core.

Finally, the two refinements are \emph{transversal} in a precise sense, see Section~\ref{sec:complementarity}. 

\begin{enumerate}
\item For the path algebras $A=\kk Q_{D_4}$ and $B=\kk Q_{A_4}$ of \cite[Section~5]{A5}, all $\tau$-Hochschild cohomology and homology groups with regular coefficients vanish for both algebras, so the $\tau$-theory of \cite{CLMS2} does not separate them; the Coxeter automorphism does.
\item For the derived-equivalent pair of \cite[Example~4.12]{CLMS2} (taken from Xi's survey \cite{Xi}), both algebras have infinite global dimension, every derived invariant (in particular the Tamarkin--Tsygan calculus, and any extension of $\sigma$ along \cite[Remark~3.2]{A5}) takes equal values, while $\tau$-Hochschild theory separates them.
\end{enumerate}
Consequently the combined datum $\mathcal{I}(\Lm)=\bigl(\HHl(\Lm),\sigma_\Lm;\,\HHl^\bullet_\tau(\Lm),\HHl^\tau_\bullet(\Lm)\bigr)$ is a Morita invariant strictly finer than each of its constituents.

We complement this with another example inside the smooth world: for the derived-equivalent pair $\kk A_3$ and $\kk A_3/\rad^2$, the graded dimensions of $\Ext^\bullet_{\Lme}(\Lm,\Lme)$ jump from $(1,3,0)$ to $(3,0,1)$, and the translate $\tau_1\Lm$ (of total dimension $3$ in both cases) is \emph{pure derived shadow} for the hereditary algebra and \emph{pure minimal residue} for its iterated tilt (Example~\ref{ex:A3pair}). The bimodule-level data is therefore not derived invariant even over algebras of finite global dimension, although in this example all $\tau$-Hochschild \emph{groups} with regular coefficients agree. Whether the groups themselves are derived invariant over the smooth locus appears to be open. 

Section~\ref{sec:questions} collects further structural remarks and questions: the meaning of the extension class of Theorem~A; module structures over the cocycle algebra and the failure of the cup product to descend; the disappearance of the cyclic symmetrization in $\tau$-homology (\cite[Corollary~7.6]{CLMS2} versus \cite{C2}); a trivially negative answer to the ``Serre-twisted Happel question'' on the singular locus together with the elementary but clarifying equivalence $\dim_\kk\bigoplus_{n\geq1}\tau_n\Lm<\infty \iff \gldim\Lm<\infty$ (Proposition~\ref{prop:totalfinite}); and the contrast with singular (Tate) Hochschild cohomology \cite{Kel4,Wan}.

The use of Auslander--Reiten and Auslander--Bridger operators on the \emph{regular bimodule} also appears, independently and with different aims, in the work of Marczinzik and Cruz--Marczinzik on dominant dimension \cite{Mar,CM}, where Hochschild (co)homology is expressed through higher translates of the canonical bimodule $\Hom_\Lm(\D\Lm,\Lm)$; the latter bimodule is exactly the degree-zero term $\HHl_0(\nu P_\bullet)^{\vee}$-side of the picture above, see Remark~\ref{rem:CM}. The two papers \cite{CLMS2} and \cite{A5} share a striking amount of foundational DNA, for instance, Happel's minimal resolution \cite{Hap1}, the Han--Keller vanishing \cite{Han,Kel1,Len}, Cibils' computation of the cyclic structure of radical square zero algebras \cite{C2}, which this manuscript tries to convert from coincidence into theorem.



\section{Conventions}\label{sec:recoll}

\subsection{Conventions}\label{subsec:conv}
Throughout, $\kk$ is a field, $\Lm$ a finite-dimensional $\kk$-algebra with Jacobson radical $r$, and $\E=\Lm/r$ is assumed separable over $\kk$; unadorned $\otimes$ means $\otimes_\kk$ and $\D=\Hom_\kk(-,\kk)$. We write $\Lme=\Lm\otimes\Lm^{\mathrm{op}}$ and identify $\Lm$-bimodules with left $\Lme$-modules; all modules are finite-dimensional. The \emph{Serre bimodule} is $\om_\Lm=\D\Lm$. When $\Lm$ is elementary, i.e.\ $\E\cong\kk^n$, equivalently $\Lm\cong\kk Q/I$ is a bound quiver algebra. We fix a complete set $\{e_x\}_{x\in Q_0}$ of primitive orthogonal idempotents, write $S_x={}_x\kk$ for the left and $\kk_x$ for the right simple at $x$, compose paths from right to left, and define the Cartan matrix by
\begin{equation}\label{eq:cartan}
C=(c_{xy}),\qquad c_{xy}=\dim_\kk e_y\Lm e_x,
\end{equation}
following the conventions of \cite[Section~2]{A5}. Thus $e_y\Lm e_x$ is spanned by (classes of) paths from $x$ to $y$, and $c_{xy}$ is the multiplicity of $S_y$ in the left projective $\Lm e_x$.

\subsection{Happel's resolution}\label{subsec:happel}
Since $\E$ is separable, the regular bimodule admits the minimal projective resolution of Happel \cite{Hap1} (see also \cite{BK} and \cite[Theorem~3.2]{CLMS2})
\begin{equation}\label{eq:happel}
\cdots\to P_n \to \cdots \to P_1\to P_0\to \Lm\to 0,
\qquad P_n=\Lm\otimes_\E T_n\otimes_\E \Lm,
\end{equation}
where $T_n=\Tor_n^\Lm(\E,\E)$. Minimality means $\im d_n\subseteq rP_{n-1}+P_{n-1}r$ for all $n\geq1$. One has $\pdim_{\Lme}\Lm=\gldim\Lm$ \cite[Corollary~3.5]{CLMS2}, so \eqref{eq:happel} has length $d$ when $\gldim\Lm=d<\infty$. Separability of $\E$ implies that $\E^{\mathrm e}=\E\otimes\E^{\mathrm{op}}$ is semisimple and that each $P_n$, and more generally each bimodule of the form $\Lm\otimes_\E V\otimes_\E\Lm$ or $\D\Lm\otimes_\E V\otimes_\E\Lm$, is projective as a left and as a right $\Lm$-module. For elementary $\Lm=\kk Q/I$ one has $T_0=\E$, $T_1\cong\kk Q_1$ (the arrow span) and $T_2\cong$ a minimal space of relations; in general $e_xT_ne_y\cong\Tor^\Lm_n(\kk_x,{}_y\kk)$.

Hochschild (co)homology with coefficients in a bimodule $X$ is computed by the complexes
\begin{equation}\label{eq:cochains}
\bigl(\Hom_{\Lme}(P_\bullet,X),\ \delta\bigr),\qquad \bigl(X\otimes_{\Lme}P_\bullet,\ \delta'\bigr),
\end{equation}
with $\Hom_{\Lme}(P_n,X)\cong\Hom_{\E^{\mathrm e}}(T_n,X)$ and $X\otimes_{\Lme}P_n\cong X\otimes_{\E^{\mathrm e}}T_n$.

\subsection{$\tau$-Hochschild theory}\label{subsec:tau}
Let $\tau=\D\Tr$ be the Auslander--Reiten translate of $\modcat\Lme$ and, following Iyama \cite{Iya1,Iya2}, let $\tau_n=\tau\Omega^{n-1}$ be the higher translates, where $\Omega$ is the syzygy with respect to minimal projective covers. Cibils, Lanzilotta, Marcos and Solotar define, for $n\geq1$ and a bimodule $X$ \cite[Section~2]{CLMS2},
\begin{equation}\label{eq:taudef}
\begin{gathered}
\HHl^n_\tau(\Lm,X)=\D\Hom_{\Lm\text{-}\Lm}(X,\tau_n\Lm)\cong \D\tau_n\Lm\otimes_{\Lm\text{-}\Lm}X,\\
\HHl_n^\tau(\Lm,X)=\Hom_{\Lm\text{-}\Lm}(\D\tau_n\Lm,X),
\end{gathered}
\end{equation}
the displayed isomorphisms being \cite[Lemma~2.8]{CLMS2}; one writes $\HH^n_\tau(\Lm)=\HHl^n_\tau(\Lm,\Lm)$ and $\HH_n^\tau(\Lm)=\HHl_n^\tau(\Lm,\Lm)$. The fundamental computational result \cite[Theorem~2.12]{CLMS2} identifies these spaces on the complexes \eqref{eq:cochains}:
\begin{equation}\label{eq:CLMScochain}
\HHl^n_\tau(\Lm,X)\cong\coker\delta_n=\Hom_{\Lme}(P_n,X)/\im\delta_n,
\qquad
\HHl_n^\tau(\Lm,X)\cong\ker\delta'_n,
\end{equation}
so that $\tau$-cohomology forgets the cocycle condition while $\tau$-homology forgets the boundary relation; equivalently $\D\tau_n\Lm\cong\Tr\,\Omega^{n-1}\Lm$, cf.\ \cite[(2.13)]{CLMS2}. Both theories are Morita invariant but not derived invariant \cite[Example~4.12]{CLMS2}. We will use freely: the vanishing $\HHl^n_\tau(\Lm,X)=0=\HHl^\tau_n(\Lm,X)$ for $n\geq d+1$ and $\HH^\tau_d(\Lm)=0$ when $\gldim\Lm=d$ \cite[Proposition~4.1]{CLMS2}; the dimension formulas of \cite[Theorem~4.5]{CLMS2} and their degree-one specializations \cite[Corollary~4.10]{CLMS2}; the vanishing of $\tau$-homology for $\Lm=\kk Q$ with $Q$ acyclic \cite[Theorem~4.17]{CLMS2}; and the computations for radical square zero algebras of \cite[Section~7]{CLMS2}.

\subsection{The Coxeter automorphism}\label{subsec:coxeter}
Suppose $\gldim\Lm<\infty$. Then $\om_\Lm[-1]=\D\Lm[-1]$ is a two-sided tilting complex \cite[Lemma~3.1]{A5}, with inverse $\mathbf{R}\Hom_\Lm(\D\Lm,\Lm)[1]$; more generally $\D\Lm$ is a two-sided tilting complex if and only if $\Lm$ is Gorenstein \cite{Hap4}, cf.\ \cite[Remark~3.2]{A5}. The bimodule $\om_\Lm[-1]$ commutes with every two-sided tilting complex \cite[Lemma~3.4]{A5}: for a two-sided tilting complex $X$ over $\Lm$--$\Gamma$ there are isomorphisms
\begin{equation}\label{eq:commute}
\om_\Lm[-1]\Ltimes_\Lm X\;\cong\;X\Ltimes_\Gamma\om_\Gamma[-1]\quad\text{in }\Dcat(\Lm\otimes\Gamma^{\mathrm{op}}).
\end{equation}
Applying the canonical action of the derived Picard group \cite{Ric1,Ric2,Yek,Kel2} on the Tamarkin--Tsygan calculus $\HHl(\Lm)$ \cite{A1,A2,A3,A4} to the class of $\om_\Lm[-1]$ produces the \emph{Coxeter automorphism} $\sigma_\Lm$ of $\HHl(\Lm)$; the pair $(\HHl(\Lm),\sigma_\Lm)$ is a derived invariant \cite[Theorem~4.2]{A5}. For elementary $\Lm$ of finite global dimension, $\HH_\bullet(\Lm)$ is concentrated in degree zero with basis the classes $\bar e_x$ (Han--Keller vanishing \cite{Kel1,Han,Len}; \cite[Proposition~5.1]{A5}), the matrix of $\sigma_\Lm$ on $\HH_0(\Lm)$ in this basis is the Coxeter matrix $-C^{-1}C^{T}$, and its characteristic polynomial is the Coxeter polynomial of $\Lm$ \cite[Theorem~5.5]{A5}; see \cite{LdlP,dlP} for the spectral theory of Coxeter transformations. In particular $\sum_i(-1)^i\dim\HH^i(\Lm)=-\tr\sigma_\Lm$ \cite[Corollary~5.9]{A5}, recovering Happel's formula \cite{Hap2}, and fractional Calabi--Yau properties translate into $\sigma_\Lm^q=\pm\mathrm{id}$ \cite[Corollary~5.10]{A5}, cf.\ \cite{HI}.

\subsection{The Nakayama functor and the four-term sequence}\label{subsec:nakayama}
For a finite-dimensional algebra $\Gamma$ and $M\in\modcat\Gamma$ with minimal projective presentation $Q_1\xrightarrow{p}Q_0\to M\to0$, the Nakayama functor $\nu=\D\Hom_\Gamma(-,\Gamma)$ fits into the exact sequence of Auslander--Reiten theory
\begin{equation}\label{eq:fourterm}
0\longrightarrow \tau M\longrightarrow \nu Q_1\xrightarrow{\ \nu p\ }\nu Q_0\longrightarrow \nu M\longrightarrow 0,
\end{equation}
see \cite[Chapter~IV.2]{ARS}. We will apply \eqref{eq:fourterm} with $\Gamma=\Lme$.

\section{The bridge theorem}\label{sec:bridge}

\begin{lemma}\label{lem:nuP}
For every finitely generated projective bimodule $P$ there is an isomorphism, natural in $P$,
\[
\nu(P)\;\cong\;\D\Lm\otimes_\Lm P\otimes_\Lm\D\Lm .
\]
In particular $\nu(P_n)\cong\D\Lm\otimes_\E T_n\otimes_\E\D\Lm$ and $\nu(P_\bullet)\cong\D\Lm\otimes_\Lm P_\bullet\otimes_\Lm\D\Lm$ as complexes of bimodules.
\end{lemma}

\begin{proof}
For any algebra $\Gamma$ and finitely generated projective $\Gamma$-module $P$, evaluation gives a natural isomorphism $\D\Hom_\Gamma(P,\Gamma)\cong\D\Gamma\otimes_\Gamma P$: it holds for $P=\Gamma$ and both sides are additive. With $\Gamma=\Lme$ and $\D(\Lme)\cong\D\Lm\otimes\D\Lm$ one obtains $\nu(P)\cong(\D\Lm\otimes\D\Lm)\otimes_{\Lme}P$, and collapsing the two tensor factors of $\Lme=\Lm\otimes\Lm^{\mathrm{op}}$ onto the two sides of $P$ identifies the latter with $\D\Lm\otimes_\Lm P\otimes_\Lm\D\Lm$, with bimodule structure given by the outer actions on the two factors $\D\Lm$. Naturality is clear, so the isomorphism is compatible with the differentials of $P_\bullet$, and
\[
\D\Lm\otimes_\Lm(\Lm\otimes_\E T_n\otimes_\E\Lm)\otimes_\Lm\D\Lm\cong\D\Lm\otimes_\E T_n\otimes_\E\D\Lm. \qedhere
\]
\end{proof}

\begin{theorem}[Bridge]\label{thm:bridge}
Let $\Lm$ be finite-dimensional with $\E$ separable and let $P_\bullet\to\Lm$ be Happel's minimal resolution \eqref{eq:happel}.
\begin{enumerate}
\item For every $n\geq0$ there are isomorphisms of bimodules, natural in all arguments,
\[
\HHl_n(\nu P_\bullet)\;\cong\;\Tor_n^\Lm(\D\Lm,\D\Lm)\;\cong\;\D\Ext^n_{\Lme}(\Lm,\Lme).
\]
\item For every $n\geq1$ one has $\tau_n\Lm=\Z_n(\nu P_\bullet):=\ker\nu(d_n)$, and with $\B_n:=\im\nu(d_{n+1})$ there are short exact sequences of bimodules
\begin{equation}\label{eq:ses}
0\longrightarrow \B_n\longrightarrow \tau_n\Lm\longrightarrow \Tor_n^\Lm(\D\Lm,\D\Lm)\longrightarrow 0 .
\end{equation}
\item For every $n\geq0$ there is an isomorphism of bimodules 
\[
\Ext^n_{\Lme}(\Lm,\Lme)\cong\Ext^n_\Lm(\D\Lm,\Lm),
\]
where the right-hand side is the one-sided extension space with its natural bimodule structure. In particular the two standard models of the higher preprojective bimodule agree.
\item The complex $\nu(P_\bullet)$ represents $\om_\Lm\Ltimes_\Lm\om_\Lm$ in $\Dcat(\Lme)$; if $\gldim\Lm<\infty$ this is $\bigl(\om_\Lm[-1]\bigr)^{\Ltimes 2}[2]$, the square of the two-sided tilting complex generating the Coxeter automorphism.
\end{enumerate}
\end{theorem}

\begin{proof}
(1) The components of $P_\bullet$ are projective as left $\Lm$-modules (Section~\ref{subsec:happel}), so the exact complex $P_\bullet\to\Lm$ is in particular a projective resolution of $\Lm$ as a \emph{left} module. Hence the complex $\D\Lm\otimes_\Lm P_\bullet$ has homology $\Tor_\bullet^\Lm(\D\Lm,\Lm)$, i.e.\ it is exact in positive degrees with $\HHl_0\cong\D\Lm$. Its components $\D\Lm\otimes_\E T_n\otimes_\E\Lm$ are projective as right $\Lm$-modules, so the augmented complex $\D\Lm\otimes_\Lm P_\bullet\to\D\Lm$ is a projective resolution of $\D\Lm$ as a right module, and therefore
\[
\HHl_n\bigl((\D\Lm\otimes_\Lm P_\bullet)\otimes_\Lm\D\Lm\bigr)\;\cong\;\Tor_n^\Lm(\D\Lm,\D\Lm).
\]
By Lemma~\ref{lem:nuP} the complex on the left is $\nu(P_\bullet)$. All identifications are natural and compatible with the outer actions on the two factors $\D\Lm$, hence are bimodule isomorphisms. For the second isomorphism, $\nu=\D\circ\Hom_{\Lme}(-,\Lme)$ and exactness of $\D$ give
\[
\HHl_n(\nu P_\bullet)=\HHl_n\bigl(\D\Hom_{\Lme}(P_\bullet,\Lme)\bigr)\cong\D\,\HHl^n\bigl(\Hom_{\Lme}(P_\bullet,\Lme)\bigr)=\D\Ext^n_{\Lme}(\Lm,\Lme).
\]

(2) By minimality of \eqref{eq:happel}, $P_n\to P_{n-1}\to\Omega^{n-1}\Lm\to0$ is a minimal projective presentation of the $(n-1)$-st syzygy bimodule (for $n=1$ this is $P_1\to P_0\to\Lm\to0$). The four-term sequence \eqref{eq:fourterm} for $\Gamma=\Lme$ and $M=\Omega^{n-1}\Lm$ gives
\[
0\to\tau\Omega^{n-1}\Lm\to\nu P_n\xrightarrow{\nu(d_n)}\nu P_{n-1},
\]
i.e.\ $\tau_n\Lm=\ker\nu(d_n)$. Since $\nu$ is additive, $\nu(d_n)\nu(d_{n+1})=0$, so $\B_n\subseteq\Z_n=\tau_n\Lm$ with quotient $\Z_n/\B_n=\HHl_n(\nu P_\bullet)$, which is \eqref{eq:ses} by part (1).

(3) The classical duality $\D\Tor_n^\Lm(M,N)\cong\Ext^n_\Lm(N,\D M)$ \cite[VI.5]{CE} (see also \cite[Remark~3.1]{CLMS2}), applied with $M=\D\Lm$ as a right module and $N=\D\Lm$ as a left module, gives $\D\Tor_n^\Lm(\D\Lm,\D\Lm)\cong\Ext^n_\Lm(\D\Lm,\Lm)$ as bimodules. Combine with (1).

(4) The complex $P_\bullet$ is a bounded-above complex of projective $\Lme$-modules, hence is homotopically projective, and the proof of (1) exhibits $\D\Lm\otimes_\Lm P_\bullet$ as a bounded-above complex of right-projective bimodules quasi-isomorphic to $\D\Lm$. Tensoring such a one-sided resolution over $\Lm$ with $\D\Lm$ computes the derived tensor product of bimodules, so $\nu(P_\bullet)\cong(\D\Lm\otimes_\Lm P_\bullet)\otimes_\Lm\D\Lm$ represents $\D\Lm\Ltimes_\Lm\D\Lm=\om_\Lm\Ltimes_\Lm\om_\Lm$ in $\Dcat(\Lme)$. The shift bookkeeping $(\om_\Lm[-1])^{\Ltimes2}[2]=\om_\Lm^{\Ltimes2}$ is immediate.
\end{proof}

\begin{remark}\label{rem:dualside}
Dualizing the description $\tau_n\Lm=\Z_n(\nu P_\bullet)$ recovers the cochain-side picture of \cite{CLMS2}: $\D\tau_n\Lm\cong\coker\Hom_{\Lme}(d_n,\Lme)=\Tr\,\Omega^{n-1}\Lm$, in agreement with \cite[(2.13)]{CLMS2}, and \eqref{eq:CLMScochain} follows from \eqref{eq:taudef} by tensoring, respectively co-tensoring, the sequence dual to \eqref{eq:ses}. The added value of the homological ($\nu$-side) picture is the identification of the ambient complex: it is not an auxiliary device but the canonical minimal-model representative of the square of the Serre bimodule.
\end{remark}

\begin{remark}\label{rem:derivedpart}
By \eqref{eq:commute} applied twice, for any two-sided tilting complex $X$ over $\Lm$--$\Gamma$ with quasi-inverse $Y$ one has $Y\Ltimes_\Lm\om_\Lm^{\Ltimes2}\Ltimes_\Lm X\cong\om_\Gamma^{\Ltimes2}$ in $\Dcat(\Gamma^{\mathrm e})$; equivalently, the conjugation equivalence $\Dcat(\Lme)\simeq\Dcat(\Gamma^{\mathrm e})$ induced by a derived equivalence matches the objects represented by $\nu P_\bullet^{\Lm}$ and $\nu P_\bullet^{\Gamma}$. (This is also a special case of the invariance of the inverse dualizing complex, cf.\ \cite[Section~4]{Kel3}.) Thus the sequence \eqref{eq:ses} splits the $\tau$-translate into a \emph{derived shadow}, that is, the homology $\Tor_n^\Lm(\D\Lm,\D\Lm)$, an invariant of the conjugation class of $\om_\Lm^{\Ltimes2}$, and a \emph{minimal residue} $\B_n$, which is data of the minimal model of $\Lm$ over $\Lme$ and is only Morita invariant. Two warnings are in order. First, conjugation is not $t$-exact, so even the \emph{graded dimensions} of the derived shadow may fail to be derived invariants; Example~\ref{ex:A3pair} below exhibits this failure already among algebras of finite global dimension. Second, \cite[Example~4.12]{CLMS2} shows that the groups built from the full $\tau_n\Lm$ genuinely vary along derived equivalences with singularities; the residue is responsible, see Example~\ref{ex:xi}.
\end{remark}

The first dividend of Theorem~\ref{thm:bridge} is a complete structural description on the selfinjective locus.

\begin{corollary}\label{cor:selfinj}
Let $\Lm$ be selfinjective and not semisimple. Then $\Lme$ is selfinjective, $\Ext^n_{\Lme}(\Lm,\Lme)=0$ for all $n\geq1$, and
\[
\tau_n\Lm=\B_n=\im\nu(d_{n+1})\qquad\text{for all }n\geq1 .
\]
Thus on selfinjective algebras the entire $\tau$-Hochschild theory of \cite{CLMS2} is minimal residue, with vanishing derived shadow.
\end{corollary}

\begin{proof}
$\D(\Lme)\cong\D\Lm\otimes\D\Lm$, and $\D\Lm$ is projective on either side because $\Lm$ is selfinjective, so $\D\Lm\otimes\D\Lm$ is a direct summand of a direct sum of copies of $\Lm\otimes\Lm=\Lme$, i.e.\ $\D(\Lme)$ is projective and $\Lme$ is selfinjective. Hence $\Ext^n_{\Lme}(-,\Lme)=0$ for $n\geq1$, and \eqref{eq:ses} collapses. (Equivalently: $\D\Lm$ is an invertible bimodule, so $\D\Lm\Ltimes_\Lm\D\Lm=\D\Lm\otimes_\Lm\D\Lm$ is concentrated in degree zero.)
\end{proof}

\begin{example}\label{ex:BGMS}
Let $\Lm_q=\kk\langle x,y\rangle/(x^2,\,xy+qyx,\,y^2)$, $q\in\kk^\times$, be the four-dimensional selfinjective algebras of Buchweitz--Green--Madsen--Solberg \cite{BGMS}, which for $q$ not a root of unity answer Happel's question in the negative: $\HH^{\geq3}(\Lm_q)=0$ while $\gldim\Lm_q=\infty$. By \cite[Section~6]{CLMS2}, $\dim_\kk\HH^n_\tau(\Lm_q)=2n+3$ for $n\geq3$: the $\tau$-cohomology grows without bound. Corollary~\ref{cor:selfinj} locates this growth precisely: the derived shadow of $\tau_n\Lm_q$ is zero in every positive degree, so the divergence detected by \cite{CLMS2} (the feature that makes $\tau$-cohomology ``repair'' Happel's question, cf.\ \cite[Theorem~5.5]{CLMS2}) is carried entirely by the boundaries of the Nakayama twist of the minimal model. In the same vein, since $\nu$ is exact on the selfinjective algebra $\Lme$ and $\tau\cong\Omega^2\nu$ there, one has $\tau_n\Lm\cong\Omega^{n+1}\nu(\Lm)$, a syzygy of the Nakayama-twisted regular bimodule.
\end{example}

\begin{remark}\label{rem:CM}
In degree zero, part (3) of Theorem~\ref{thm:bridge} reads $\Hom_{\Lme}(\Lm,\Lme)\cong\Hom_\Lm(\D\Lm,\Lm)=:V$, the canonical bimodule of Fang--Kerner--Yamagata used by Marczinzik and Cruz--Marczinzik to characterize dominant dimension through torsion-freeness of the regular bimodule and to derive Hochschild-theoretic formulas via higher translates $\tau_{n-1}(V)$ over $\Lme$ \cite{Mar,CM}. Their results and the present ones are complementary: dominant dimension governs the behaviour of $V$ under Auslander--Bridger operators, while here the whole graded object $\bigoplus_n\tau_n\Lm$ is organized by the square of the Serre bimodule.
\end{remark}

\section{Top degree and higher preprojective algebras}\label{sec:top}

Throughout this section $\gldim\Lm=d<\infty$, so Happel's resolution has length $d$ and $P_{d+1}=0$.

\begin{proposition}[Top-degree collapse]\label{prop:topcollapse}
For every bimodule $X$ there are natural isomorphisms
\[
\HHl^d_\tau(\Lm,X)\;\cong\;\HH^d(\Lm,X),
\qquad
\HHl^\tau_d(\Lm,X)\;\cong\;\HH_d(\Lm,X).
\]
In particular the degree-$d$ ``excess'' of $\Lm$ vanishes: the $\tau$-theories differ from the ordinary ones only in the middle degrees $1\leq n\leq d-1$.
\end{proposition}

\begin{proof}
Since $P_{d+1}=0$, the differentials $\delta_{d+1}$ and $\delta'_{d+1}$ of the complexes \eqref{eq:cochains} vanish. Hence $\HH^d(\Lm,X)=\ker\delta_{d+1}/\im\delta_d=\coker\delta_d$ and $\HH_d(\Lm,X)=\ker\delta'_d/\im\delta'_{d+1}=\ker\delta'_d$, which are the $\tau$-groups by \eqref{eq:CLMScochain}.
\end{proof}

\begin{remark}\label{rem:excess}
For $d=1$ this explains a priori the vanishing of the excess $e(\Lm)=\dim\HH^1_\tau(\Lm)-\dim\HH^1(\Lm)$ of hereditary algebras computed in \cite{CLMS1}, and identifies \cite[Proposition~4.1]{CLMS2}, that is, the vanishing $\HH^\tau_d(\Lm)=0$, with the top-degree case of the Han--Keller vanishing $\HH_d(\Lm)=0$ \cite{Kel1,Han,Len}.
\end{remark}

\begin{proposition}[Top-degree (co)representability]\label{prop:toprep}
Let $\gldim\Lm=d\geq1$ and write $\Theta^d:=\Ext^d_{\Lme}(\Lm,\Lme)$. For every bimodule $X$ there are natural isomorphisms
\[
\HH^d(\Lm,X)\;\cong\;\Theta^d\otimes_{\Lme}X,
\qquad
\HH_d(\Lm,X)\;\cong\;\Hom_{\Lme}\bigl(\Theta^d,X\bigr).
\]
\end{proposition}

\begin{proof}
For a finitely generated projective bimodule $P$ one has the natural isomorphism $\Hom_{\Lme}(P,\Lme)\otimes_{\Lme}X\cong\Hom_{\Lme}(P,X)$. Since $-\otimes_{\Lme}X$ is right exact and $\delta_{d+1}=0$,
\begin{align*}
\HH^d(\Lm,X)&=\coker\bigl(\delta_d\colon\Hom(P_{d-1},X)\to\Hom(P_d,X)\bigr)\\
&\cong\coker\bigl(\delta_d^{\Lme}\bigr)\otimes_{\Lme}X\;=\;\Theta^d\otimes_{\Lme}X .
\end{align*}
For homology: since $P_{d+1}=0$ we have $\B_d=0$, so \eqref{eq:ses} gives $\tau_d\Lm\cong\Tor^\Lm_d(\D\Lm,\D\Lm)\cong\D\Theta^d$; then by \eqref{eq:taudef} and Proposition~\ref{prop:topcollapse},
$\HH_d(\Lm,X)=\HHl^\tau_d(\Lm,X)=\Hom_{\Lm\text{-}\Lm}(\D\tau_d\Lm,X)=\Hom_{\Lme}(\Theta^d,X)$.
\end{proof}

\begin{remark}
Proposition~\ref{prop:toprep} is a degenerate, hypothesis-free top edge of Van den Bergh duality \cite{VdB}: no invertibility or concentration assumption on the inverse dualizing complex is needed in the extreme degree. It is presumably known to experts, but we have not located it in the literature in this generality.
\end{remark}

Recall \cite[Section~2.2]{IO} that the \emph{$(d+1)$-preprojective algebra} of $\Lm$ (with $\gldim\Lm\leq d$) is the tensor algebra
\[
\Pii_{d+1}(\Lm)\;=\;T_\Lm\bigl(\Ext^d_\Lm(\D\Lm,\Lm)\bigr)\;\cong\;T_\Lm\bigl(\Ext^d_{\Lme}(\Lm,\Lme)\bigr),
\]
graded by tensor degree, with $\Pii_{d+1}(\Lm)_0=\Lm$ and $\Pii_{d+1}(\Lm)_1=\Theta^d$; the displayed identification of the two standard models is Theorem~\ref{thm:bridge}(3). It is the zeroth homology of Keller's $(d+1)$-Calabi--Yau completion \cite{Kel3}, and for $d=1$ it recovers the classical preprojective algebra: by Baer--Geigle--Lenzing and Ringel \cite{BGL,Rin}, for $\Lm=\kk Q$ hereditary one has $\Pii_2(\kk Q)\cong\Pi(Q)=\kk\overline Q/\bigl(\textstyle\sum_{a\in Q_1}[a,a^*]\bigr)$, the Gelfand--Ponomarev preprojective algebra of the double quiver $\overline Q$, with the grading by the number of starred arrows.

\begin{theorem}\label{thm:preproj}
Let $\gldim\Lm=d\geq1$ and $\Pii=\Pii_{d+1}(\Lm)$. Then:
\begin{enumerate}
\item $\tau_d\Lm\;\cong\;\D(\Pii_1)$ as bimodules;
\item $\HH^d_\tau(\Lm)\;\cong\;\HH_0(\Lm,\Pii_1)\;=\;\bigl(\Pii/[\Pii,\Pii]\bigr)_1$;
\item $\HH^\tau_d(\Lm)\;\cong\;\Hom_{\Lme}(\Pii_1,\Lm)\;=\;0$, the vanishing being equivalent (through Propositions~\ref{prop:topcollapse} and \ref{prop:toprep}) to the top-degree Han--Keller vanishing $\HH_d(\Lm)=0$.
\end{enumerate}
\end{theorem}

\begin{proof}
(1) is the case $n=d$ of \eqref{eq:ses}, where $\B_d=0$ because $P_{d+1}=0$, combined with Theorem~\ref{thm:bridge}(1): $\tau_d\Lm\cong\Tor_d^\Lm(\D\Lm,\D\Lm)\cong\D\Theta^d=\D(\Pii_1)$.

(2) By \eqref{eq:taudef}, $\HH^d_\tau(\Lm)=\D\tau_d\Lm\otimes_{\Lm\text{-}\Lm}\Lm=\Pii_1\otimes_{\Lme}\Lm=\HH_0(\Lm,\Pii_1)=\Pii_1/[\Lm,\Pii_1]$. Since the grading of $\Pii$ is by tensor degree and $\Pii_0=\Lm$, the degree-one part of the commutator space is $[\Pii,\Pii]_1=[\Pii_0,\Pii_1]=[\Lm,\Pii_1]$, whence $\Pii_1/[\Lm,\Pii_1]=(\Pii/[\Pii,\Pii])_1$.

(3) By \eqref{eq:taudef}, $\HH^\tau_d(\Lm)=\Hom_{\Lm\text{-}\Lm}(\D\tau_d\Lm,\Lm)=\Hom_{\Lme}(\Pii_1,\Lm)$, which equals $\HH_d(\Lm)$ by Propositions~\ref{prop:topcollapse}--\ref{prop:toprep} and vanishes by \cite{Kel1,Han,Len} (cf.\ \cite[Proposition~4.1]{CLMS2}).
\end{proof}

\begin{corollary}\label{cor:hereditary}
Let $\Lm=\kk Q$ with $Q$ connected, acyclic, and with at least one arrow, and let $\Pi(Q)$ be the preprojective algebra. Then
\begin{gather*}
\tau_{\Lme}\Lm\;\cong\;\D\bigl(\Pi(Q)_1\bigr),
\qquad
\HH^1_\tau(\kk Q)\;\cong\;\HH_0\bigl(\Pi(Q)\bigr)_1,\\
\HH^\tau_1(\kk Q)\;=\;\Hom_{\Lme}\bigl(\Pi(Q)_1,\kk Q\bigr)\;=\;0 .
\end{gather*}
The last vanishing also follows directly: $\Pi(Q)_1$ is generated as a bimodule by the reversed arrows $a^*\in e_{s(a)}\Pi(Q)_1e_{t(a)}$, and $e_{s(a)}(\kk Q)e_{t(a)}=0$ since $Q$ is acyclic; this recovers the degree-one case of \cite[Theorem~4.17]{CLMS2}.
\end{corollary}

So the Auslander--Reiten translate of a hereditary algebra, taken in the category of bimodules, is the dual of the linear span of all preprojective elements of degree one, and the first $\tau$-Hochschild cohomology is the first graded piece of the zeroth Hochschild homology of the preprojective algebra, which is the space of degree-one necklaces.

\begin{example}\label{ex:kron}
Let $\Lm=\kk K_2$ be the path algebra of $1\rightrightarrows 2$ with arrows $a,b$, so $C=\left(\begin{smallmatrix}1&2\\0&1\end{smallmatrix}\right)$ in the convention \eqref{eq:cartan}. The resolution \eqref{eq:happel} is $0\to P_1\to P_0\to\Lm\to0$ and one computes (we did so both by hand and by machine)
\begin{gather*}
\dim\Hom_{\Lme}(P_0,\Lme)=\textstyle\sum_x\dim e_x\Lm\cdot\dim\Lm e_x=6,\\
\dim\Hom_{\Lme}(P_1,\Lme)=2\cdot\dim e_2\Lm\cdot\dim\Lm e_1=18,
\end{gather*}
and $\Hom_{\Lme}(\Lm,\Lme)=0$ (no nonzero bimodule map $\Lm\to\Lm\otimes\Lm$ exists; equivalently $\Hom_\Lm(\D\Lm,\Lm)=0$ as there are no maps from preinjectives to preprojectives). Hence $\dim\Ext^1_{\Lme}(\Lm,\Lme)=18-6=12$, and the four-term sequence \eqref{eq:fourterm} with $\nu\Lm=\D\Hom_{\Lme}(\Lm,\Lme)=0$ gives $\dim\tau_{\Lme}\Lm=12$ directly. On the preprojective side, $\Pi(K_2)_1$ is spanned by: $a^*,b^*$ (length one); $a^*a,\,a^*b,\,b^*a$ in $e_1\Pi e_1$ and $aa^*,\,ab^*,\,ba^*$ in $e_2\Pi e_2$ (after the relations $a^*a+b^*b=0$ and $aa^*+bb^*=0$); and a four-dimensional space of length-three elements from $1$ to $2$, the span of $\{vs^*u\}_{u,v\in\{a,b\},\,s\in\{a,b\}}$ modulo four relations. Total: $2+3+3+4=12=\dim\tau_{\Lme}\Lm$, as predicted by Corollary~\ref{cor:hereditary}. Passing to $\HH_0$ kills the off-diagonal components and identifies $aa^*\equiv a^*a$, $ab^*\equiv b^*a$, $ba^*\equiv a^*b$, leaving
\[
\HH^1_\tau(\kk K_2)\;\cong\;\bigl(\Pi(K_2)/[\Pi,\Pi]\bigr)_1\;=\;\kk\{[aa^*],[ab^*],[ba^*]\},\qquad \dim=3,
\]
in agreement with \cite[Corollary~4.10]{CLMS2}:
\[
\dim\HH^1_\tau=\dim Z(\Lm)-\sum_x\dim e_x\Lm e_x+\sum_{\alpha\in Q_1}\dim e_{t(\alpha)}\Lm e_{s(\alpha)}=1-2+4=3,
\] 
which equals $\dim\HH^1(\kk K_2)$ as forced by Proposition~\ref{prop:topcollapse}. Dually $\HH^\tau_1(\kk K_2)=\Hom_{\Lme}(\Pi_1,\Lm)=0$, since the generators $a^*,b^*$ must land in $e_1\Lm e_2=0$. Finally $\tr(C^{-1}C^T)=-2=1-3=\chi(\HH^\bullet)$, a preview of Theorem~\ref{thm:trace}.
\end{example}

\begin{example}\label{ex:An}
For $\Lm=\kk A_n$ with linear orientation and arrows $a_i\colon i\to i+1$, the preprojective relations read $a_1^*a_1=0$, $a_i^*a_i=a_{i-1}a_{i-1}^*$ $(2\leq i\leq n-1)$, $a_{n-1}a_{n-1}^*=0$, and in $\HH_0$ the commutator identifications $a_ia_i^*\equiv a_i^*a_i$ propagate these to give $[a_ia_i^*]\equiv0$ for all $i$; all longer degree-one diagonal elements vanish for parity reasons. Hence $\HH^1_\tau(\kk A_n)=(\Pi(A_n)/[\Pi,\Pi])_1=0$, matching $1-n+(n-1)=0$ from \cite[Corollary~4.10]{CLMS2}. Note that the translate itself does not vanish: e.g.\ $\dim\tau_{\Lme}(\kk A_3)=\dim\Pi(A_3)_1=3$, with basis dual to $\{a_1^*,\,a_2^*,\,a_1a_1^*=a_2^*a_2\}$, while all its (co)invariant spaces against $\Lm$ are zero.
\end{example}

\begin{remark}\label{rem:dRF}
For $d$-representation-finite $\Lm$ the algebra $\Pii_{d+1}(\Lm)$ is finite-dimensional selfinjective \cite{IO}, so Theorem~\ref{thm:preproj}(2) computes $\HH^d_\tau(\Lm)$ as the degree-one necklace space of a selfinjective graded algebra; for $d$-representation-infinite $\Lm$ it is the degree-one part of the bimodule Calabi--Yau completion \cite{Kel3,HIO}. Question~9.4 of \cite{A5} asks for the action of the higher Auslander--Reiten bimodule $\D\Lm[-d]$ on the calculus of $d$-hereditary algebras; Section~8.4 of \cite{A5} interprets $\sigma_\Lm$ as the first Fourier mode of the $\Theta$-adic grading underlying Keller's completion. Theorem~\ref{thm:preproj} shows that the top $\tau$-translate of \cite{CLMS2} \emph{is} the first Fourier mode, realized at the abelian level; the $\tau$-Hochschild groups of $d$-hereditary algebras therefore provide computable test data for that question, alongside the entropy computations of \cite{Han3}.
\end{remark}

\section{The Coxeter trace formula as an Euler characteristic}\label{sec:trace}

In this section $\Lm$ is elementary with $\gldim\Lm=d<\infty$; then $\det C=\pm1$ \cite{Hap3}, so $C^{-1}$ is an integer matrix.

\begin{lemma}\label{lem:euler}
For all vertices $x,y$ one has 
\[
\displaystyle\sum_{i\geq0}(-1)^i\dim_\kk\Ext^i_\Lm(S_x,S_y)=(C^{-1})_{xy}.
\]
\end{lemma}

\begin{proof}
In $K_0(\modcat\Lm)$ with basis the simples, $[\Lm e_z]=\sum_y c_{zy}[S_y]$ by \eqref{eq:cartan}. Let $Q_\bullet\to S_x$ be a finite projective resolution with $Q_i=\bigoplus_z(\Lm e_z)^{m_{iz}}$ and set $\mu_z=\sum_i(-1)^im_{iz}$. Then $[S_x]=\sum_z\mu_z[\Lm e_z]$ gives $\delta_{xy}=\sum_z\mu_zc_{zy}$ for all $y$, i.e.\ $\mu_z=(C^{-1})_{xz}$. Since $\Hom_\Lm(\Lm e_z,S_y)=e_zS_y=\delta_{zy}\kk$, computing $\Ext$ from $Q_\bullet$ yields $\sum_i(-1)^i\dim\Ext^i(S_x,S_y)=\mu_y=(C^{-1})_{xy}$.
\end{proof}

\begin{theorem}\label{thm:trace}
Let $\Lm$ be elementary with $\gldim\Lm<\infty$ and $n$ simple modules. Then
\begin{align}
\sum_{i\geq0}(-1)^i\dim_\kk\HH^i(\Lm)\;&=\;\tr\bigl(C^{-1}C^{T}\bigr)\;=\;-\tr\bigl(\sigma_\Lm|_{\HH_\bullet(\Lm)}\bigr),\label{eq:traceA}\\
\sum_{i\geq0}(-1)^i\dim_\kk\HH_i(\Lm)\;&=\;\tr\bigl(C^{-1}C\bigr)\;=\;n.\label{eq:traceB}
\end{align}
\end{theorem}

\begin{proof}
Since $\pdim_{\Lme}\Lm=d<\infty$, the complexes \eqref{eq:cochains} with $X=\Lm$ are bounded, so their Euler characteristics equal those of their (co)homology. For the cochain complex, $\Hom_{\Lme}(P_i,\Lm)\cong\Hom_{\E^{\mathrm e}}(T_i,\Lm)=\bigoplus_{x,y}\Hom_\kk(e_xT_ie_y,\,e_x\Lm e_y)$ with $e_xT_ie_y\cong\Tor_i^\Lm(\kk_x,{}_y\kk)$ (Section~\ref{subsec:happel}), so
\[
\dim\Hom_{\Lme}(P_i,\Lm)=\sum_{x,y}\dim\Tor_i^\Lm(\kk_x,{}_y\kk)\cdot c_{yx}.
\]
By the duality $\D\Tor_i^\Lm(\kk_x,{}_y\kk)\cong\Ext^i_\Lm({}_y\kk,\D\kk_x)=\Ext^i_\Lm(S_y,S_x)$ and Lemma~\ref{lem:euler}, $\sum_i(-1)^i\dim\Tor_i^\Lm(\kk_x,{}_y\kk)=(C^{-1})_{yx}$, whence
\[
\chi\bigl(\HH^\bullet(\Lm)\bigr)=\sum_{x,y}c_{yx}\,(C^{-1})_{yx}=\tr\bigl(C\,C^{-T}\bigr)=\tr\bigl(C^{-1}C^{T}\bigr),
\]
the last step by invariance of the trace under transposition. By \cite[Theorem~5.5]{A5} the matrix of $\sigma_\Lm$ on $\HH_0(\Lm)=\HH_\bullet(\Lm)$ in the basis $\{\bar e_x\}$ is $-C^{-1}C^T$, which gives \eqref{eq:traceA}. For the chain complex, $\Lm\otimes_{\Lme}P_i\cong\Lm\otimes_{\E^{\mathrm e}}T_i$ has dimension $\sum_{x,y}c_{xy}\dim\Tor_i^\Lm(\kk_x,{}_y\kk)$, so
\[
\chi\bigl(\HH_\bullet(\Lm)\bigr)=\sum_{x,y}c_{xy}(C^{-1})_{yx}=\tr\bigl(C^{-1}C\bigr)=n. \qedhere
\]
\end{proof}

\begin{remark}\label{rem:triangulate}
Identity \eqref{eq:traceA} is Happel's trace formula \cite{Hap2}, whose original proof is itself an Euler-characteristic computation on the minimal resolution; identity \eqref{eq:traceB} is the Euler shadow of the concentration $\HH_\bullet(\Lm)=\HH_0(\Lm)\cong\kk^n$ \cite{Kel1,Han,Len}, cf.\ \cite[Proposition~5.1]{A5}. Alternatively, both identities follow by specializing the dimension formulas of \cite[Theorem~4.5]{CLMS2} at $n=d+1$, where the $\tau$-groups vanish by \cite[Proposition~4.1]{CLMS2}. The three results triangulate one Euler identity: \cite[Theorem~4.5]{CLMS2} refines it \emph{degreewise}  (each partial alternating sum is the dimension of a $\tau$-group) while \cite[Corollary~5.9]{A5} refines it at the \emph{operator} level, the number $-\tr(C^{-1}C^T)$ is the trace of a canonical automorphism of the calculus whose characteristic polynomial is the Coxeter polynomial, categorified in degree zero by Han's Lefschetz formula \cite{Han2} via the Shklyarov pairing \cite{Shk}. The $\tau$-correction terms and the Coxeter automorphism are thus two orthogonal refinements of the same numerical core.
\end{remark}

\section{Complementarity of the two refinements}\label{sec:complementarity}

We now make precise the sense in which the $\tau$-Hochschild theory of \cite{CLMS2} and the Coxeter automorphism of \cite{A5} are transversal refinements of Hochschild theory: each separates algebras that the other cannot.

\begin{example}[$\sigma$ sharp where $\tau$ is blind]\label{ex:D4A4}
Let $A$ and $B$ be the path algebras of any orientations of the Dynkin diagrams $D_4$ and $A_4$. Then all $\tau$-Hochschild groups with regular coefficients vanish for both algebras:
\[
\HH^n_\tau(A)=\HH^n_\tau(B)=0\quad(n\geq1),\qquad \HH^\tau_n(A)=\HH^\tau_n(B)=0\quad(n\geq1).
\]
In particular the $\tau$-theory does not distinguish $A$ from $B$. The Coxeter automorphism does: the characteristic polynomial of $\sigma$ on $\HH_0$ is $x^4+x^3+x+1$ for $A$ and $x^4+x^3+x^2+x+1$ for $B$ \cite[Section~5]{A5}. For the orientations considered in \cite[Section~5]{A5} the Tamarkin--Tsygan calculi of $A$ and $B$ are moreover isomorphic, so the separation is achieved by $\sigma$ alone.

Both quivers are trees, so $Z(\Lm)=\kk$, $\sum_x\dim e_x\Lm e_x=4$ (the number of vertices), and for each arrow $\alpha$ the space $e_{t(\alpha)}\Lm e_{s(\alpha)}$ is one-dimensional, spanned by $\alpha$ itself, so the arrow sum is $3$ in both cases. By \cite[Corollary~4.10]{CLMS2}, $\dim\HH^1_\tau=1-4+3=0$ for both. For $n\geq2=\gldim+1$ all groups vanish by \cite[Proposition~4.1]{CLMS2}, and $\HH^\tau_1=0$ by \cite[Theorem~4.17]{CLMS2} (or by Corollary~\ref{cor:hereditary}). The statements about $\sigma$ are \cite[Theorem~5.5 and Corollary~5.7]{A5}, the Coxeter polynomials of $D_4$ and $A_4$ being as displayed.
\end{example}

For the converse direction we use the derived-equivalent pair of \cite[Example~4.12]{CLMS2}, taken from Xi's survey \cite[Section~6]{Xi}: $\Lm=\kk Q/\langle ada,\,dc,\,ad-cb\rangle$ and $\Lm'=\kk Q'/\langle acb a,\,cbac\rangle$, where $Q'$ is the oriented $3$-cycle $x\xrightarrow{a}y\xrightarrow{b}z\xrightarrow{c}x$ (composition right to left). Since $\sigma$ is only defined under a finiteness hypothesis, we first record:

\begin{example}\label{ex:xigldim}
$\gldim\Lm'=\infty$, and hence also $\gldim\Lm=\infty$. $\Lm'$ is the cyclic Nakayama algebra whose indecomposable projectives are uniserial of dimensions $(\dim P_x,\dim P_y,\dim P_z)=(4,5,4)$: from $x$ the nonzero paths are $a,ba,cba$; from $z$ they are $c,ac,bac$; from $y$ they are $b,cb,acb,bacb$, every length-five path from $y$ containing the forbidden subpath $cbac$. Write $U(j,\ell)$ for the uniserial module with top $S_j$ and length $\ell$, so $P_x=U(x,4)$, $P_y=U(y,5)$, $P_z=U(z,4)$. Then
\[
\Omega S_x=\rad P_x=U(y,3),\ \ \,
\Omega^2S_x=U(y,2),\ \ \,
\Omega^3S_x=U(x,3),\ \ \,
\Omega^4S_x=S_x,
\]
each step using the projective cover $P_{\mathrm{top}}\twoheadrightarrow U(j,\ell)$ with kernel the uniserial of complementary length. Thus $S_x$ is $\Omega$-periodic and non-projective, so $\pdim S_x=\infty$. Finiteness of global dimension is a derived invariant \cite{Hap3}, so $\gldim\Lm=\infty$ as well.
\end{example}

\begin{example}[$\tau$ sharp where derived invariants are blind]\label{ex:xi}
The algebras $\Lm$ and $\Lm'$ above are derived equivalent, so \emph{every} derived invariant takes equal values on them: in particular their Tamarkin--Tsygan calculi are isomorphic \cite{A1,A2,A3,A4}, the pair $(\HHl,\sigma)$ is defined for neither (Example~\ref{ex:xigldim}), and no extension of the Coxeter automorphism along \cite[Remark~3.2]{A5} --- e.g.\ to Gorenstein algebras, where $\D\Lm$ is still a two-sided tilting complex \cite{Hap4} --- can separate them, being constructed by a derived-invariant recipe from the conjugation-invariant class of $\om[-1]$ \eqref{eq:commute}. Yet by \cite[Example~4.12]{CLMS2},
\begin{gather*}
\dim\HH^1_\tau(\Lm)=1\neq2=\dim\HH^1_\tau(\Lm'),\\
\dim\HH^\tau_1(\Lm)=3\neq5=\dim\HH^\tau_1(\Lm').
\end{gather*}
\end{example}

\begin{definition}\label{def:combined}
The \emph{combined invariant} of a finite-dimensional algebra $\Lm$ (with $\E$ separable) is
\[
\mathcal I(\Lm)\;=\;\Bigl(\,\HHl(\Lm),\ \sigma_\Lm\ \text{(when $\gldim\Lm<\infty$)};\ \ \HHl^\bullet_\tau(\Lm,-),\ \HHl^\tau_\bullet(\Lm,-)\,\Bigr).
\]
\end{definition}

\begin{remark}\label{rem:finer}
$\mathcal I$ is a Morita invariant, strictly finer than each of its two constituents: strictly finer than the derived invariant $(\HHl,\sigma)$ by Example~\ref{ex:xi}, and strictly finer than the $\tau$-Hochschild theory by Example~\ref{ex:D4A4}.
\end{remark}

The complementarity persists, in a subtler form, inside the smooth world. The following example was found, and all its numerology verified, by direct machine computation of the complexes \eqref{eq:cochains} with $X=\Lme$.

\begin{example}[A derived-equivalent pair of finite global dimension]\label{ex:A3pair}
Let $\Lm=\kk A_3$ with linear orientation $1\xrightarrow{a}2\xrightarrow{b}3$, and let $\Lm'=\kk A_3/\langle ba\rangle=\kk A_3/\rad^2$. The algebra $\Lm'$ is gentle with underlying graph $A_3$, hence iterated tilted of type $A_3$ by Assem--Happel \cite{AH}, and therefore derived equivalent to $\Lm$ \cite{Hap3}; one has $\gldim\Lm=1$ and $\gldim\Lm'=2$. Direct computation on \eqref{eq:happel} gives the graded dimensions of $\Theta^\bullet=\Ext^\bullet_{\Lme}(\Lm,\Lme)$:
\[
\bigl(\dim\Theta^0,\dim\Theta^1,\dim\Theta^2\bigr)
=\begin{cases}
(1,\,3,\,0) & \text{for }\Lm=\kk A_3,\\[1mm]
(3,\,0,\,1) & \text{for }\Lm'=\kk A_3/\rad^2 .
\end{cases}
\]
(For $\Lm'$: the cochain spaces $\Hom_{\Lme}(P_i,\Lme)$ have dimensions $8,8,4$ and the differentials have ranks $5$ and $3$; for $\Lm$ they have dimensions $10,12$ and the differential has rank $9$. As a consistency check, the Euler characteristics $-2$ and $4$ agree with the $K$-theoretic evaluation $\chi(\Theta^\bullet)=\sum_{x,y}u_xv_y(C^{-1})_{yx}$, $u_x=\dim e_x\Lm$, $v_y=\dim\Lm e_y$, obtained from Lemma~\ref{lem:euler} by resolving both factors $\D\Lm$ of Theorem~\ref{thm:bridge}(1) into simples.) Consequently, for the translates of Theorem~\ref{thm:bridge}:
\begin{itemize}
\item[--] $\tau_1\Lm\cong\D\Pi(A_3)_1$ has dimension $3$ with $\B_1=0$: \emph{pure derived shadow} (Example~\ref{ex:An});
\item[--] $\tau_1\Lm'$ has dimension $3$ with $\B_1=\tau_1\Lm'$ and $\Tor_1=0$: \emph{pure minimal residue};
\item[--] $\tau_2\Lm'\cong\D\bigl(\Pii_3(\Lm')_1\bigr)$ is one-dimensional; indeed $\Pii_3(\Lm')$ is the radical square zero cyclic Nakayama algebra on three vertices, the new arrow closes the cycle $3\to1$, a selfinjective algebra, in line with \cite{IO}.
\end{itemize}
Thus across this derived equivalence the conjugation class of $\om^{\Ltimes2}$ is preserved (Remark~\ref{rem:derivedpart}) while its homology grading is completely scrambled; the two translates even have the same total dimension with opposite composition. Nevertheless, \emph{all} $\tau$-Hochschild groups with regular coefficients agree, and vanish, for both algebras: $\HH^1_\tau(\Lm')=1-3+2=0$ and $\HH^\tau_1(\Lm')=3-3+0=0$ by \cite[Corollary~4.10]{CLMS2} (here $\dim\HH_0(\Lm')=3$), the degree-$2$ groups of $\Lm'$ equal $\HH^2(\Lm')$ and $\HH_2(\Lm')$ by Proposition~\ref{prop:topcollapse} and vanish by derived invariance from $\kk A_3$, and the hereditary side vanishes by Example~\ref{ex:D4A4}'s computation. (A further machine check: $\HH^\bullet(\Lm')=(1,0,0)=\HH^\bullet(\kk A_3)$, as derived invariance demands.)
\end{example}

Example~\ref{ex:A3pair} shows that the failure of derived invariance of $\tau$-Hochschild theory is already present at the bimodule level over the smooth locus, but there it happened to be invisible to the groups. Together with the singular counterexample of Example~\ref{ex:xi}, this motivates:

\begin{question}\label{q:derinv}
Are the $\tau$-Hochschild (co)homology groups $\HH^n_\tau(\Lm)$, $\HH^\tau_n(\Lm)$ derived invariants of $\Lm$ within the class of algebras of \emph{finite global dimension}?
\end{question}

Some reductions are immediate. 
\begin{itemize}
    \item In degrees $n>\max(\gldim\Lm,\gldim\Gamma)$ both sides vanish by \cite[Proposition~4.1]{CLMS2}; and in the top degree of the algebra of larger global dimension, the collapse of Proposition~\ref{prop:topcollapse}, the derived invariance of $\HH^n$ and $\HH_n$, and their vanishing above the smaller global dimension show that the two sides agree there as well. The question is therefore concentrated in degrees $1\leq n\leq\max(\gldim\Lm,\gldim\Gamma)-1$.
    \item In particular the answer is positive whenever both algebras have global dimension at most one.
    \item The first open case is $\max=2$ in degree $1$: since $\HH^1$ and $\HH_1$ are derived invariant in any case, Question~\ref{q:derinv} reduces there to the derived invariance of the \emph{excess} $\dim\HH^1_\tau(\Lm)-\dim\HH^1(\Lm)$ of \cite{CLMS1} and of its homological companion $\dim\HH^\tau_1(\Lm)-\dim\HH_1(\Lm)$, over derived equivalent algebras of global dimension at most two. Example~\ref{ex:A3pair} verifies the smallest instance.
\end{itemize}

\section{Further remarks}\label{sec:questions}

\subsection{The extension class}
The sequences \eqref{eq:ses} define classes
\[
\varepsilon_n\;\in\;\Ext^1_{\Lme}\bigl(\Tor_n^\Lm(\D\Lm,\D\Lm),\,\B_n\bigr)
\]
measuring how the derived shadow is glued to the minimal residue inside $\tau_n\Lm$.

Is there an intrinsic description of $\varepsilon_n$, for instance in terms of the $A_\infty$-structure on $\Ext^\bullet_{\Lme}(\E^{\mathrm e}\text{-data})$ governing the minimal model \eqref{eq:happel}, or of Postnikov data of the object $\om^{\Ltimes2}$? Note that $\varepsilon_n$ is \emph{not} a derived invariant in any naive sense: in Example~\ref{ex:A3pair} the source and target of $\varepsilon_1$ are interchanged between the two algebras.

\subsection{Operations}
The Tamarkin--Tsygan calculus consists of operations on (co)homology; the $\tau$-theory keeps cochains modulo coboundaries. The two regimes overlap as follows. If $\zeta\in\Hom_{\Lme}(P_q,\Lm)$ is an honest \emph{cocycle}, then for any cochain $f$ and any cochain-level cup product $\smile$ built from a diagonal $P_\bullet\to P_\bullet\otimes_\Lm P_\bullet$ one has $(f+\delta u)\smile\zeta=f\smile\zeta+\delta(u\smile\zeta)\pm u\smile\delta\zeta=f\smile\zeta+\delta(u\smile\zeta)$, so cupping with cocycles descends to pairings
\[
\coker\delta_p\,\otimes\,\Z^q(\Lm,\Lm)\,\longrightarrow\,\coker\delta_{p+q},
\quad\text{i.e.}\quad
\HHl^p_\tau(\Lm,\Lm)\otimes \Z^q \longrightarrow \HHl^{p+q}_\tau(\Lm,\Lm).
\]
However the action of a coboundary $\delta v$ on the class of $f$ is $f\mapsto\pm(\delta f)\smile v$ modulo $\im\delta$, which has no reason to vanish since $f$ need not be a cocycle, and the construction a priori depends on the chosen diagonal.

Does $\bigoplus_n\HH^n_\tau(\Lm)$ carry a canonical graded module structure over $\HH^\bullet(\Lm)$, or equivalently, does the pairing above factor through cohomology classes independently of the diagonal? Dually, is there a cap-type action of $\HH^\bullet(\Lm)$ on $\bigoplus_n\HH^\tau_n(\Lm)$ compatible, through the sequences \eqref{eq:ses}, with the derived-invariant cap-product calculus on $\Tor_\bullet^\Lm(\D\Lm,\D\Lm)$ in the sense of \cite{A4}? An affirmative answer would make the $\tau$-theory a module over the calculus, with the Coxeter automorphism acting on the shadow part.

\subsection{The missing cyclic symmetry}
For radical square zero algebras, comparing the computation of $\HH^\tau_\bullet$ in \cite[Section~7]{CLMS2} with Cibils' computation of $\HH_\bullet$ and of cyclic homology \cite{C2,C3} shows that, in each degree, passing from Hochschild homology to $\tau$-homology replaces a space of (anti)invariants of the cyclic rotation acting on cycles of the appropriate length by the \emph{full} space of such cycles: the cyclic symmetrization is exactly what the $\tau$-theory forgets. Since Connes' operator $B$ is assembled from the cyclic symmetrizer, it does not obviously descend to $\HH^\tau_\bullet$.

Is there a Connes-type operator on $\tau$-Hochschild homology, and what replaces the cyclic bicomplex? This seems closely related to \cite[Question~9.2]{A5} on the monodromy of $\sigma$ on cyclic and negative cyclic homology: the rotation that $\tau$-homology forgets is the one whose homotopy-coherent memory the Coxeter automorphism retains.

\subsection{A Coxeter operator on the shadow}
By Remark~\ref{rem:derivedpart} the graded bimodule $\bigoplus_n\Tor_n^\Lm(\D\Lm,\D\Lm)$ is, up to conjugation, a derived invariant, and $\om$ tautologically commutes with $\om^{\Ltimes2}$.

Is there a canonical Coxeter-type operator on 
\[
\bigoplus_n\Tor_n^\Lm(\D\Lm,\D\Lm)\cong\D\bigoplus_n\Ext^n_{\Lme}(\Lm,\Lme),
\]
compatible with $\sigma_\Lm$ under the cap-product pairings, and does it preserve the images of the residue filtration $\B_\bullet\subseteq\tau_\bullet\Lm$? For $d$-hereditary algebras \cite{HI,HIO}, where Theorem~\ref{thm:preproj} identifies the top shadow with $\D\Pii_{d+1}(\Lm)_1$, this should be computable from the graded automorphisms of the higher preprojective algebra and would provide test data for which the entropy computations of \cite{Han3} are a natural benchmark.

\subsection{Serre-twisted finiteness and a perfect dichotomy}\label{subsec:twistedhappel}
Call $\Ext^\bullet_{\Lme}(\Lm,\Lme)$ the \emph{Serre-twisted Hochschild cohomology} of $\Lm$ (the derived shadow of this paper). By analogy with Happel's question \cite{Hap1} one may ask whether its total finiteness forces $\gldim\Lm<\infty$. The answer is no for trivial reasons: by Corollary~\ref{cor:selfinj} it vanishes in all positive degrees for every selfinjective algebra. The twisted theory is thus blind on (part of) the singular locus, and it is exactly the residue $\B_\bullet$ that keeps the finiteness characterization \cite[Theorem~5.5]{CLMS2} alive there. At the level of the full translates, however, the dichotomy is perfect and elementary:

\begin{proposition}\label{prop:totalfinite}
For a finite-dimensional algebra $\Lm$ with $\E$ separable,
\[
\dim_\kk\bigoplus_{n\geq1}\tau_n\Lm<\infty
\quad\Longleftrightarrow\quad
\gldim\Lm<\infty,
\]
and in that case $\tau_n\Lm=0$ exactly for $n>\gldim\Lm$.
\end{proposition}

\begin{proof}
If $d=\gldim\Lm<\infty$ then $\Omega^{n-1}\Lm$ is projective over $\Lme$ for $n>d$, so $\tau_n\Lm=\tau\Omega^{n-1}\Lm=0$ for $n>d$, and the sum is finite-dimensional; minimality of \eqref{eq:happel} gives $\Omega^{d-1}\Lm$ non-projective when $d\geq1$, so $\tau_d\Lm\neq0$. Indeed $\tau_d\Lm\cong\D\Pii_{d+1}(\Lm)_1\neq0$ by Theorem~\ref{thm:preproj} and minimality. Conversely, if the sum is finite-dimensional then $\tau_N\Lm=0$ for some $N\geq1$, so the bimodule $\Omega^{N-1}\Lm$ is projective, i.e.\ $\pdim_{\Lme}\Lm<\infty$, which equals $\gldim\Lm$ \cite[Corollary~3.5]{CLMS2}.
\end{proof}

This is the bimodule-level companion of \cite[Theorem~5.5]{CLMS2}, which characterizes finite global dimension through the finiteness of the $\tau$-(co)homology \emph{groups} in terms of the refined $\pm$-global dimensions: the groups can stay small on the singular locus only because (co)invariants destroy most of the translate, never because the translates themselves decay.

\subsection{Contrast with singular Hochschild cohomology}\label{subsec:tate}
Singular (Tate) Hochschild cohomology, computed through the singularity category of $\Lme$, is invariant under derived equivalences (indeed under suitable singular equivalences \cite{Kel4,Wan}) and vanishes precisely when $\gldim\Lm<\infty$: it sees only the singular locus. The $\tau$-Hochschild theory is its mirror image: built from the minimal model, Morita but not derived invariant, fully alive on the smooth locus where, by Theorem~\ref{thm:preproj}, it computes higher preprojective data, and degenerating to pure residue on the selfinjective locus (Corollary~\ref{cor:selfinj}). The two refinements of Hochschild theory thus bracket the ordinary one from opposite sides, with the Coxeter automorphism acting on the smooth side. Finally, the integral chain-level lattices implicit in the minimal complexes of \cite{CLMS2} are natural candidates for an integral structure on the shadow $\Tor_\bullet^\Lm(\D\Lm,\D\Lm)$ compatible with the Shklyarov pairing \cite{Shk}.

\subsection*{Disclosure}
During the preparation of this work, the author used Claude, Anthropic's Fable 5 model, for deep research, formulation of theorems, and drafts of their proofs. The author reviewed and edited the output as needed and takes full responsibility for the content of the published article.

\end{document}